\documentclass[12pt]{article}
\usepackage[fleqn]{amsmath}
\usepackage[cp1251]{inputenc}
\usepackage[english]{babel}
\usepackage{amsmath,amsfonts,amssymb,graphicx,makeidx,amsthm}
\begin{document}

\begin{center}
    \textbf{The general form of the Euler--Poisson--Darboux equation and application of transmutation method}
\end{center}

\begin{center}
    \textbf{Elina L. Shishkina, Sergei M. Sitnik}
\end{center}

\textbf{Keywords:} {Bessel operator, Euler--Poisson--Darboux equation, Hankel transform;

\begin{abstract}
In the paper we find solution representations in the compact integral form  to the  Cauchy problem for  a general form of the Euler--Poisson--Darboux equation with Bessel operators via generalized translation and spherical mean operators for all values of the parameter $k$, including also not studying before exceptional odd negative values.  We use a Hankel transform method to prove results in a unified way. Under additional conditions we prove that a distributional solution is a classical one too.  A transmutation property for connected generalized spherical mean is proved and importance of applying transmutation methods for differential equations with Bessel operators is emphasized. The paper also contains a short historical introduction on differential equations with Bessel operators and a rather detailed reference list of monographs and papers on mathematical theory and applications of this class of   differential equations.
\end{abstract}

\numberwithin{equation}{section}
\newtheorem{theorem}{Theorem}[section]
\newtheorem{lemma}{Lemma}[section]
\newtheorem{corollary}{Corollary}[section]
\allowdisplaybreaks

\section{Introduction}

The classical Euler--Poisson--Darboux (EPD)  equation is defined by
\begin{equation}\label{ClEPD}
\frac{{\partial }^2u}{\partial t^2}+\frac{k}{t}\frac{\partial u}{\partial t}=\sum\limits^n_{i=1}\frac{{\partial }^2u}{\partial x^2_i},\quad u=u(x,t;k),\quad x\in\mathbb{R}^n,\quad t>0,\quad -\infty<k<\infty.
\end{equation}
The operator acting by variable $t$ in \eqref{ClEPD} is  the \textbf{Bessel operator} and we will denote it (see, for example, \cite{Kipr}, p. 3) as
$$
(B_k)_t=\frac{{\partial }^2}{\partial t^2}+\frac{k}{t}\frac{\partial }{\partial t}.
$$
When $n=1$ the equation \eqref{ClEPD} appears in Leonard Euler's work (see \cite{Euler} p. 227) and later was studied by Sim${\rm \acute{e}}$on Denis Poisson in  \cite{Poisson}, by Gaston Darboux in \cite{Darboux} and by Bernhard Riemann in \cite{Riman}.

For the Cauchy problem initial conditions to the solution of equation \eqref{ClEPD} are added
\begin{equation}\label{USL0}
    u(x,0;k)=f(x),\qquad \frac{\partial u(x,t;k)}{\partial t}\biggr|_{t=0}=0.
\end{equation}

An interest to the multidimensional equation \eqref{ClEPD} has increased significantly after Alexander Weinstein's papers \cite{Weinstein0}--\cite{Weinstein4}. In \cite{Weinstein0}--\cite{Weinstein1} the Cauchy problem  for  \eqref{ClEPD} is considered with  $k\in\mathbb{R}$, the first initial condition being non--zero and the second initial condition equals to zero.
A solution of the Cauchy problem \eqref{ClEPD}--\eqref{USL0}  in the classical sense was obtained in \cite{Weinstein0}--\cite{Young} and in the distributional sense in \cite{Bresters2}, \cite{CSh}.
S.A. Tersenov in \cite{Tersenov} solved  the Cauchy problem  for  \eqref{ClEPD} in the general form where the first and the second conditions are non--zeros.
Different problems for the equation \eqref{ClEPD} with many applications to gas dynamics, hydrodynamics, mechanics, elasticity and plasticity and so on  were also studied in \cite{Aks}, \cite{Blum}--\cite{CSh}, \cite{Chap}, \cite{Diaz}--\cite{Dza2}, \cite{Fal}, \cite{Fran1}--\cite{Fran3}, \cite{Gordeev}, \cite{Hairul}, \cite{KuFr}--\cite{Mizes}, \cite{Pul1}--\cite{Pul2}, \cite{Smirnov}--\cite{Vekua1}, \cite{Young}, \cite{ZhTr}, and  of course the above list of references is incomplete.

In this article we consider the singular with respect to all variables hyperbolic differential equation, which is a generalization of multidimensional Euler--Poisson--Darboux (EPD) equation \eqref{ClEPD}:
\begin{equation}\label{EPD0}
\frac{{\partial }^2u}{\partial t^2}+\frac{k}{t}\frac{\partial u}{\partial t}=(\triangle_\gamma)_x\ u,\quad \quad u=u(x,t;k), \quad k\in\mathbb{R},\quad t>0,
\end{equation}
with the  singular elliptic operator defined by
\begin{equation}\label{LapBess1}
(\triangle_\gamma)_x=\sum\limits_{i=1}^{n} (B_{\gamma_i})_{x_i}=\sum\limits_{i=1}^{n}\left( \frac{\partial ^{2} }{\partial x_{i}^{2} } +\frac{\gamma _{i} }{x_{i} } \frac{\partial}{\partial x}\right)=\sum\limits_{i=1}^{n}\frac{1}{x_i^{\gamma_i}}\frac{\partial}{\partial
x_i}x_i^{\gamma_i}\frac{\partial}{\partial x_i}
\end{equation}
under the next natural restrictions
$$
\gamma_i>0, \qquad  x=(x_1,...,x_n),\qquad x_i>0,\qquad i=1,2,...,n,
$$
together with  initial conditions \eqref{USL0} as above.

We will call the equation \eqref{EPD0} as the \textbf{Euler--Poisson--Darboux equation in the general form}.

Let us specially emphasize that singular differential equations with the operator \eqref{LapBess1} including equations \eqref{ClEPD} and \eqref{EPD0} were thoroughly studied in many papers by  I.A.~Kipriyanov's school, the results are partially systemized in his monograph \cite{Kipr}. In accordance with I.A.~Kipriyanov's terminology the operator \eqref{LapBess1} is classified as $B$--elliptic operator (sometimes also the term Laplace--Bessel operator is used), and equations \eqref{ClEPD} and  \eqref{EPD0} are classified as $B$--hyperbolic equations.
In connection with results of this scientific school let us mention  papers  of L.A.~Ivanov \cite{Ivan1}--\cite{Ivan2}, \cite{Ivan3} in which important problems for EPD equation were solved, such as generalizations to homogeneous symmetric Riemann spaces, energy equipartition property, equations with a product of EPD--type multipliers. Also note papers \cite{Barabash}, \cite{Shi5}, \cite{LPSh2} on application of spherical mean and generalized translation operators, generalized mean value theorems. Differential equations of EPD type are applied in the study of fractional powers of EPD, generalized EPD operators and connected generalized Riesz--type potentials, cf. \cite{Sh1}--\cite{Sh4}.

Another important approach to differential equations with Bessel operators is based on application of the transmutation theory. This method is essential in the study of singular problems with use of special classes of transmutations such as Sonine, Poisson, Buschman--Erd\'elyi ones and different forms of fractional integrodifferential operators, cf. \cite{CSh}--\cite{Car3}, \cite{Sit6}--\cite{Sit8}, \cite{Kra}, \cite{Sit1}--\cite{Sit5}. Abstract differential equations with Bessel operators were studied in and in fact were mostly initiated by the famous monograph \cite{CSh}, cf. also recent papers \cite{Glushak1}--\cite{Glushak3}.

Considering the Cauchy problem \eqref{EPD0}--\eqref{USL0} in more details, David Fox in \cite{Fox} (cf. also  \cite{CSh}, p. 243 and \cite{Stellmacher}) proved solution uniqueness  for $k\geq 0$ and find a solution representation in the explicit form for all $k$ except odd negative values. The explicit solution was found via Lauricella functions in fact as $n$--times series, which is not convenient  for applications and numerical solving. In all the above references the case $k\neq -1,-3,-5,...$ was expelled and  not studied.
So in \cite{Barabash}, \cite{Shi5}, \cite{LPSh2} different approaches from those used in \cite{Fox} to the solution of this Cauchy problem were considered.

In this paper we find solution representations to the above Cauchy problem in the compact integral form via generalized translation and spherical mean operators for all values of the parameter $k$, including also not studying before exceptional odd negative values.  We use a Hankel transform method to prove results in a unified way. Under additional conditions we prove that a distributional solution is a classical one too.

 \section{Definitions and propositions}

We deal with the subset of the Euclidean space
$$
\mathbb{R}^n_+{=}\{x{=}(x_1,\ldots,x_n)\in\mathbb{R}^n,\,\,\, x_1{>}0,\ldots, x_n{>}0\}.
$$
Let denote $|x|=\sqrt{\sum\limits_{i=1}^n x_i^2}$ and $\Omega$ be finite or infinite open set in $\mathbb{R}^n$ symmetric with respect  to each hyperplane $x_i{=}0$, $i=1,...,n$, $\Omega_+=\Omega\cap{\mathbb{R}}^n_+$ and $\overline{\Omega}_+=\Omega\cap\overline{\mathbb{R}}\,\!^n_+$ where $$
\overline{\mathbb{R}}\,^n_+{=}\{x{=}(x_1,\ldots,x_n)\in\mathbb{R}^n,\,\,\, x_1{\geq}0,\ldots, x_n{\geq}0\}.
$$ We consider the class $C^m(\Omega_+)$ consisting of $m$-times differentiable on $\Omega_+$ functions
and denote by $C^m(\overline{\Omega}_+)$ the subset of functions from $C^m(\Omega_+)$ such that all derivatives  of these functions with respect to $x_i$ for any $i=1,...,n$
 are continuous up to $x_i{=}0$. Function $f\in C^m(\overline{\Omega}_+)$ we will call \emph{even with respect to} $x_i$, $i=1,...,n$ if $\frac{\partial^{2k+1}f}{\partial x_i^{2k+1}}\biggr|_{x=0}=0$ for all nonnegative integer $k\leq \frac{m-1}{2}$ (see \cite{Kipr}, p. 21). Class $C^m_{ev}(\overline{\Omega}_+)$ consists of functions from $C^m(\overline{\Omega}_+)$ even with respect to each variable $x_i$, $i=1,...,n$.
In the following we will denote $C^m_{ev}(\overline{\mathbb{R}}\,\!^n_+)$ by $C^m_{ev}$.
We set
$$
C^\infty_{ev}(\overline{\Omega}_+)=\bigcap C^m_{ev}(\overline{\Omega}_+)
$$
with intersection taken for all finite $m$ and $C^\infty_{ev}(\overline{\mathbb{R}}_+)=C^\infty_{ev}$.
Let ${\stackrel{\circ}C}\,\!^\infty_{ev}(\overline{\Omega}_+)$ be the space of all functions  $f{\in}C^\infty_{ev}(\overline{\Omega}_+)$ with a compact support. Elements of ${\stackrel{\circ}C}\,\!^\infty_{ev}(\overline{\Omega}_+)$ we will call \emph{test functions} and use the notation ${\stackrel{\circ}C}\,\!^\infty_{ev}(\overline{\Omega}_+){=}\mathcal{D}_+(\overline{\Omega}_+)$.

As the space of basic functions we will use  the subspace of the space of rapidly decreasing functions:
$$
 S_{ev}({\mathbb{R}}^{n}_+)=\left\{f\in C^\infty_{ev}:\sup _{{x\in {\mathbb{R}}^{n}_+}}\left|x^{\alpha }D^{\beta }f(x)\right|<\infty \quad \forall \alpha ,\beta \in \mathbb {Z} _{+}^{n}\right\},
$$
where $\alpha=(\alpha_1,...,\alpha_n)$, $\beta=(\beta_1,...,\beta_n)$, $\alpha_1,...,\alpha_n,\beta_1,...,\beta_n$ are integer nonnegative numbers, $x^\alpha= x_1^{\alpha_1} x_2^{\alpha_2} \ldots x_n^{\alpha_n}$,
${D}^\beta={D}^{\beta_1}_{x_1}...{D}^{\beta_n}_{x_n}$, ${D}_{x_j}=\frac{\partial}{\partial x_j}$.

We deal with multi-index $\gamma{=}(\gamma_1,{\ldots},\gamma_{n})$ consists of positive fixed reals $\gamma_i>0$, $i{=}1,{...},n$, $|\gamma|{=}\gamma_1{+}{\ldots}{+}\gamma_{n}.$
Let $L_p^{\gamma}(\Omega_+)$, $1{\leq}p{<}\infty$, be the space of  all measurable in $\Omega_+$ functions even with respect to each variable $x_i$, $i=1,...,n$ such that
$$
\int\limits_{\Omega_+}|f(x)|^p x^\gamma dx<\infty,
$$
where and further
$$
 x^\gamma=\prod\limits_{i=1}^n x_i^{\gamma_i}.
$$
For a real number $p\geq 1$, the $L_p^\gamma(\Omega_+)$--norm of $f$ is defined by
$$
||f||_{L_p^\gamma(\Omega_+)}=\left(\,\,\int\limits_{\Omega_+}|f(x)|^p x^\gamma dx\right)^{1/p}.
$$

Weighted measure of $\Omega_+$ is denoted by
  ${\rm{mes}}_\gamma(\Omega)$ and is defined by formula
$$
{\rm{mes}}_\gamma(\Omega_+)=\int\limits_{\Omega_+}x^\gamma dx.
$$
For every measurable function $f(x)$ defined on $\mathbb{R}^n_+$ we consider
$$
\mu_\gamma(f,t)={\rm{mes}}_\gamma\{x\in\mathbb{R}^n_+:\,|f(x)|>t\}=\int\limits_{\{x:\,\,|f(x)|>t\}^+}x^\gamma dx
$$
where $\{x{:}|f(x)|{>}t\}^+{=}\{x{\in}\mathbb{R}^n_+{:}|f(x)|{>}t\}$. We will call the function   $\mu_\gamma=\mu_\gamma(f,t)$ a {\it{weighted distribution function}} $|f(x)|$.

A space $L_\infty^\gamma(\Omega_+)$ is defined as a set of measurable on  $\Omega_+$ and even with respect to each variable functions $f(x)$ such as
$$
||f||_{L_\infty^\gamma(\Omega_+)}=\underset{x\in \Omega_+}{{\rm ess\,sup}_\gamma}|f(x)|
=\inf\limits_{a\in\Omega_+}\{\mu_\gamma(f,a)=0\}<\infty.
$$
For $1\leq p\leq\infty$  the $L_{p,loc}^\gamma(\Omega_+)$ is the set of functions $u(x)$ defined almost everywhere in $\Omega_+$ such that $uf\in L_{p}^\gamma(\Omega_+)$ for any $f\in{\stackrel{\circ}C}\,^\infty_{ev}(\overline{\Omega}_+)$.
Each function $u(x)\in L_{1,loc}^\gamma(\Omega_+)$ will be identified with the functional $u\in \mathcal{D}_+'(\overline{\Omega}_+)$
acting according to the formula
\begin{equation}\label{RegDist}
(u,f)_\gamma=\int\limits_{\mathbb{R}^n_+} u(x)\,f(x)\,x^\gamma\, dx,\qquad f\in {\stackrel{\circ}C}\,^\infty_{ev}(\overline{\mathbb{R}}\,^n_+).
\end{equation}
Functionals $u\in \mathcal{D}_+'(\overline{\Omega}_+)$ acting by the formula \eqref{RegDist} will be called  \emph{regular weighted functionals}.
All other functionals $u\in \mathcal{D}_+'(\overline{\Omega}_+)$ will be called  \emph{singular weighted functionals}.

We will use regular weighted functional $(t^2-|x|^2)_{+,\gamma}^\lambda$ defined by the formula
\begin{equation}\label{WF}
((t^2-|x|^2)_{+,\gamma}^\lambda,\varphi)_\gamma=\int\limits_{\{x{\in}\mathbb{R}^n_+{:}|x|<t\}}(t^2-|x|^2)^\lambda\varphi(x)x^\gamma
dx,\qquad \varphi\in S_{ev},\quad \lambda{\in}\mathbb{C}.
\end{equation}

The symbol $j_\nu$ is  used for the normalized Bessel function:
$$j_\nu(t)=\frac{2^\nu\Gamma(\nu+1)}{t^\nu}J_\nu(t),$$
where  $J_{\nu}(t)$ is the Bessel function of the first kind of order $\nu$ (see \cite{Watson}). The function $j_\nu(t)$ is even by $t$.

In our investigation we use the multidimensional  Hankel (Fourier--Bessel)  transform.
The \textbf{multidimensional Hankel transform}  of a function $f(x)$ is given by (see \cite{Bateman}):
$$
F_B[f](\xi)=(F_B)_x[f(x)](\xi)=\widehat{f}(\xi)=\int\limits_{\mathbb{R}^n_+}f(x)\,\mathbf{j}_\gamma(x;\xi)x^\gamma dx,
$$
where
$$\mathbf{j}_\gamma(x;\xi)=\prod\limits_{i=1}^n j_{\frac{\gamma_i-1}{2}}(x_i\xi_i),\qquad \gamma_1>0,...,\gamma_n>0.$$

For $f\in S_{ev}$ inverse multidimensional Hankel transform is defined by
$$
F^{-1}_B[\widehat{f}(\xi)](x)=f(x)=\frac{2^{n-|\gamma|}}{\prod\limits_{j=1}^n\,
\Gamma^2\left(\frac{\gamma_j{+}1}{2}\right)}\int\limits_{\mathbb{R}^n_+}
\mathbf{j}_\gamma(x,\xi)\widehat{f}(\xi)\xi^\gamma\:d\xi.
$$

We will deal with the \textbf{singular Bessel differential operator} $B_{\nu}$ (see, for example, \cite{Kipr}, p. 5):
$$
(B_{\nu})_{t}=\frac{\partial^2}{\partial t^2}+\frac{\nu}{t}\frac{\partial}{\partial t}=\frac{1}{t^{\nu}}\frac{\partial}{\partial
t}t^{\nu}\frac{\partial}{\partial t},\qquad t>0.
$$
and the elliptical singular operator or the Laplace-Bessel operator  $\triangle_\gamma$:
\begin{equation}\label{LapBess}
\triangle_\gamma=(\triangle_\gamma)_x=\sum\limits_{i=1}^{n} (B_{\gamma_i})_{x_i}=\sum\limits_{i=1}^{n}\left( \frac{\partial ^{2} }{\partial x_{i}^{2} } +\frac{\gamma _{i} }{x_{i} } \frac{\partial}{\partial x}\right)=\sum\limits_{i=1}^{n}\frac{1}{x_i^{\gamma_i}}\frac{\partial}{\partial
x_i}x_i^{\gamma_i}\frac{\partial}{\partial x_i}.
\end{equation}
The operator \eqref{LapBess} belongs to the class of B-elliptic operators by I. A. Kipriyanovs' classification  (see \cite{Kipr}).

The \emph{B-polyharmonic of order }$p$ function $f=f(x)$ is the function $f{\in}C^{2p}_{ev}(\overline{\mathbb{R}}\,_n^+)$ such that
\begin{equation}\label{PolG1}
\Delta^p_\gamma f=0,
\end{equation}
where $\Delta_\gamma$ is operator \eqref{LapBess}.  The operator \eqref{PolG1} was considered in \cite{Kipr}. The B{--}polyharmonic of order $1$ function we will call \emph{B--harmonic}.

Using formulas 9.1.27 from \cite{AbramowitzStegunSpF} we obtain
\begin{equation}\label{BessBess}
 (B_{\nu})_t {j}_{\frac{\nu-1}{2}}(\tau t)=-\tau^2{j}_{\frac{\nu-1}{2}}(\tau t).
\end{equation}

We will use the generalized convolution operator defined  by the formula
$$
(f*g)_\gamma=\int\limits_{\mathbb{R}^n_+}f(y)(\,^\gamma T^yg)(x)y^\gamma dy,
$$
where $^\gamma T^y$ is multidimensional generalized translation
$$^\gamma T^y=\,^{\gamma_1} T_{x_1}^{y_1}...^{\gamma_n}T_{x_n}^{y_n},
$$
each one-dimensional operator $^{\gamma_i} T_{x_i}^{y_i}$, $i=1,...,n$ acts according to (see \cite{Levitan})
$$
^{\gamma_i} T_{x_i}^{y_i}f(x){=}
\frac{\Gamma\left(\frac{\gamma_i+1}{2}\right)}{\Gamma\left(\frac{\gamma_i}{2}\right)\Gamma\left(\frac{1}{2}\right)}
\times$$
$$\times\int\limits_0^\pi f(x_1,...,x_{i-1},\sqrt{x_i^2+y_i^2-2x_iy_i\cos\alpha_i},x_{i+1},...,x_n)\,\,
\sin^{\gamma_i-1}\alpha_i\,d\alpha_i\,.
$$
Based on the multidimensional generalized translation $^\gamma T^y$  the weighted spherical mean $M^\gamma_t[f(x)]$ of a suitable function is defined by the formula
\begin{equation}\label{05}
M^\gamma_t[f(x)]=\frac{1}{|S_1^+(n)|_\gamma}\int\limits_{S^+_1(n)}\,^\gamma T_x^{t\theta}f(x)
\theta^\gamma dS,
\end{equation}
where $\theta^\gamma{=}\prod\limits_{i=1}^{n}\theta_i^{\gamma_i},$ $S^+_1(n){=}\{\theta{:}|\theta|{=}1,\theta{\in}\mathbb{R}^n_+\}$ and
$|S^+_1(n)|_\gamma=\frac{\prod\limits_{i=1}^n{\Gamma\left(\frac{\gamma_i{+}1}{2}\right)}}{2^{n-1}\Gamma\left(\frac{n{+}|\gamma|}{2}\right)}.
$
It is easy to see that
\begin{equation}\label{MeanCond}
  M^\gamma_0[f(x)]=f(x),\qquad \frac{\partial}{\partial t}M^\gamma_t[f(x)]\biggr|_{t=0}=0.
\end{equation}


\begin{lemma} Let $u\in S_{ev}$ then
\begin{equation}\label{SaprBessel}
F_B[\Delta_\gamma f](\xi)=-|\xi|^2F_B[f](\xi).
\end{equation}
\end{lemma}
\begin{proof}
We have
$$
F_B[\Delta_\gamma f](\xi)=\int\limits_{\mathbb{R}^n_+}[\Delta_\gamma f(x)]\,\mathbf{j}_\gamma(x;\xi)x^\gamma dx=
$$
$$
=\sum\limits_{i=1}^{n}\int\limits_{\mathbb{R}^n_+} \left[\frac{1}{x_i^{\gamma_i}}\frac{\partial}{\partial
x_i}x_i^{\gamma_i}\frac{\partial}{\partial x_i}f(x)\right]\,\mathbf{j}_\gamma(x;\xi)x^\gamma dx.
$$
Integrating by parts by variable $x_i$ and using formula \eqref{BessBess}, we obtain
$$
F_B[\Delta_\gamma f](\xi)=\sum\limits_{i=1}^{n}\int\limits_{\mathbb{R}^n_+} f(x)\,\left[\frac{1}{x_i^{\gamma_i}}\frac{\partial}{\partial
x_i}x_i^{\gamma_i}\frac{\partial}{\partial x_i}\mathbf{j}_\gamma(x;\xi)\right]x^\gamma dx=
$$
$$
=\sum\limits_{i=1}^{n} (-\xi_i^2)\int\limits_{\mathbb{R}^n_+} f(x)\,\mathbf{j}_\gamma(x;\xi)x^\gamma dx=-|\xi|^2\int\limits_{\mathbb{R}^n_+} f(x)\,\mathbf{j}_\gamma(x;\xi)x^\gamma dx=-|\xi|^2F_B[f](\xi).
$$
\end{proof}


\begin{lemma} We have the following formula
\begin{equation}\label{FBOF}
\frac{(F_B)_x(t^2-|x|^2)_{+,\gamma}^{\frac{k-n-|\gamma|-1}{2}}}{\Gamma\left(\frac{k-n-|\gamma|+1}{2}\right)}=\frac{t^{{k-1}}\prod\limits^n_{i=1}\Gamma\left(\frac{\gamma_i+1}{2}\right)}{\pi 2^n\Gamma\left(\frac{k+1}{2}\right)}j_{\frac{k-1}{2}}(t|x|),
\end{equation}
where $(t^2-|x|^2)_{+,\gamma}^{\frac{k-n-|\gamma|-1}{2}}$ is defined by \eqref{WF}.
\end{lemma}
The formula \eqref{FBOF} is obtained similarly to the formula (5) from \cite{Gelfand}, p. 291.

\vskip 1cm

\begin{lemma}
Let $u =u(x,t;k)$ denote the solution of \eqref{EPD0}. Then for solutions of the equation \eqref{EPD0} the next two important recursion formulas hold
\begin{equation} \label{Rec1} u(x,t;k)=t^{1-k}u(x,t;2-k), \end{equation}
\begin{equation} \label{Rec2} u_t(x,t;k)=tu(x,t;2+k). \end{equation}
\end{lemma}
It is  particular cases of A. Weinstein's formulas which state that for any equation of the form $u_{tt}+\frac{k}{t}u_t=X(u)$, in which $X$ is an operator does not depend on $t$ the relations  \eqref{Rec1} and \eqref{Rec2} hold true (see \cite{Fox}).

\vskip 1cm

\begin{lemma} The weighted spherical mean $M^\gamma_t[f(x)]$ is the transmutation operator (cf. \cite{Car1}--\cite{Car3}, \cite{Sit1}--\cite{Sit2}) intertwining $(\Delta_\gamma)_x$ and $(B_{n+|\gamma|-1})_t$ for the
$f\in C^2_{ev}$:
\begin{equation}\label{OPPr}
(B_{n+|\gamma|-1})_tM^\gamma_t[f(x)]=M^\gamma_t[(\Delta_\gamma)_xf(x)].
\end{equation}
\end{lemma}
\begin{proof} First of all we note that the  function
$f\in C^2_{ev}$  satisfies the relation
 \begin{equation}\label{EQ15}
\int\limits_{B_t^+(n)}\,f(x)\,x^\gamma\, dx=\int\limits_0^t \,\,\lambda^{n+|\gamma|-1}\,\,d\lambda \int\limits_{S_1^+(n)}f(\lambda\theta )\,\theta^\gamma dS_\theta,
 \end{equation}
which can be easy obtained by  passing to spherical coordinates  $x=\lambda \theta$, $|\theta|=1$ in the left hand of \eqref{EQ15}.
From \eqref{EQ15} we get
\begin{eqnarray}\label{EQ16}
|S^+_1(n)|_\gamma\int\limits_0^t \lambda^{n+|\gamma|-1}M^\gamma_\lambda[f(x)] d\lambda=
\int\limits_0^t \lambda^{n+|\gamma|-1}d\lambda
 \int\limits_{S^+_1(n)}(T^{\lambda y}f)(x)\,\,y^\gamma dS_y=\nonumber\\
 =\int\limits_{B^+_r(n)}(T^{z}f)(x)z^\gamma dz.
\end{eqnarray}
Let us apply the operator $\Delta_{\gamma}$ to both sides of the relation \eqref{EQ16} with respect to $x$, then we obtain
$$
|S^+_1(n)|_\gamma \sum\limits_{i=1}^n\int\limits_0^t \lambda^{n+|\gamma|-1}B_{\gamma_i}\,M^\gamma_\lambda[f(x)]d\lambda =\sum\limits_{i=1}^n\int\limits_{B^+_t(n)}(B_{\gamma_i})_{x_i}(T^{z}f)(x)z^\gamma dz=
$$
\begin{equation}\label{EQ17}
=\sum\limits_{i=1}^n\int\limits_{B^+_t(n)}(B_{\gamma_i})_{z_i}T_x^{z}f(x)z^\gamma dz.
\end{equation}
We have the next Green formula
\begin{equation}\label{EQ18}
\int\limits_{\overline{\Omega}\,^+} (v\,\Delta_\gamma w-w\,\Delta_\gamma v)\,\,x^\gamma dx=
\int\limits_{\Gamma=\partial\overline{\Omega}^+}\left(v \frac{\partial w}{\partial \vec{\nu}}-
w \frac{\partial v}{\partial \vec{\nu}}\right) x^\gamma\,d\Gamma_x\,,
\end{equation}
where $w,v\in C^2_{ev}(\overline{\Omega}\,^+)$, $\vec{\nu}$ is the outward normal to the boundary $\Gamma=\partial\overline{\Omega}\,^+$ of the  $\overline{\Omega}\,^+$.
This formula was presented in \cite{Shi5}.

 By applying formula
 \eqref{EQ18} to the right-hand side of relation \eqref{EQ17}, we obtain
$$\sum\limits_{i=1}^n\int\limits_{B^+_t(n)}(B_{\gamma_i})_{z_i}(T^{z}f)(x)z^\gamma dz=
\sum\limits_{i=1}^n\int\limits_{S^+_t(n)}\frac{\partial}{\partial z_i}(T^{z}f)(x)  \,\cos(\vec{\nu},\vec{e}_i)\,z^\gamma \,dS_z,
$$
where $\vec{e}_i$ is the direction of the axis $Oz_i$, $i=1,...,n$.

Now, by using the fact that the direction of the outward normal to the boundary of a ball
with center the origin coincides with the direction of the position vector of the point on the ball,
we obtain the relation
$$
\sum\limits_{i=1}^n\int\limits_{B^+_t(n)}(B_{\gamma_i})_{z_i}(T^{z}f)(x)z^\gamma dz= t^{n+|\gamma|-1}\int\limits_{S^+_1(n)}\frac{\partial}{\partial t}(T^{t\theta}f)(x) \,\theta^\gamma dS_\theta=$$
$$=
|S_1^+(n)|_\gamma t^{n+|\gamma|-1}\frac{\partial}{\partial t}M^\gamma_t [f(x)].
$$
Returning to \eqref{EQ17}, we obtain
\begin{equation}\label{EQ19}
\sum\limits_{i=1}^n\int\limits_0^t \lambda^{n+|\gamma|-1}B_{\gamma_i}M^\gamma_\lambda[f(x)]d\lambda =t^{n+|\gamma|-1}\frac{\partial}{\partial t}M^\gamma_t[f(x)].
\end{equation}
By differentiating relation \eqref{EQ19} with respect to $t$, we obtain
$$
\sum\limits_{i=1}^n t^{n+|\gamma|-1}B_{\gamma_i}M^\gamma_t [f(x)] =(n+|\gamma|-1)t^{n+|\gamma|-2}\frac{\partial}{\partial t}M^\gamma_t [f(x)]+t^{n+|\gamma|-1}\frac{\partial^2}{\partial t^2}M^\gamma_t [f(x)]
$$
or
$$
\sum\limits_{i=1}^n B_{\gamma_i}M^\gamma_t[ f(x)] =\frac{n+|\gamma|-1}{t}\frac{\partial}{\partial t}M^\gamma_t[f(x)]+\frac{\partial^2}{\partial t^2}M^\gamma_t[f(x)],
$$
and so
\begin{equation}\label{Eq20}
(\Delta_\gamma)_x M^\gamma_t[ f(x)] =\frac{n+|\gamma|-1}{t}\frac{\partial}{\partial t}M^\gamma_t [f(x)]+\frac{\partial^2}{\partial t^2}M^\gamma_t[ f(x)].
\end{equation}

Now let consider $(\Delta_\gamma)_x M^\gamma_t[f(x)]$. Using the commutativity of $B_{\gamma_i}$ and $T_{x_i}^{t\theta_i}$ (see \cite{Kipr}) we obtain
$$
(\Delta_\gamma)_x M^\gamma_t[f(x)]= \frac{1}{|S_1^+(n)|_\gamma}\,(\Delta_\gamma)_x\,\int\limits_{S^+_1(n)}\,^\gamma T_x^{t\theta}f(x)
\theta^\gamma dS_\theta=
$$
$$=\frac{1}{|S_1^+(n)|_\gamma}\int\limits_{S^+_1(n)}\,^\gamma T_x^{t\theta}[(\Delta_\gamma)_x f(x)]
\theta^\gamma dS_\theta= M^\gamma_t [(\Delta_\gamma)_x f(x)].
$$
which with \eqref{Eq20} gives \eqref{OPPr}.
\end{proof}
The similarly proof also can be found in \cite{Shi5}.

\section{Transmutation method}

An important and powerful approach to differential equations with Bessel operators is based on application of the transmutation theory. This method is essential in the study of singular problems with use of special classes of transmutations such as Sonine, Poisson, Buschman--Erd\'elyi ones and different forms of fractional integrodifferential operators, cf. \cite{CSh}--\cite{Car3}, \cite{Sit6}--\cite{Sit8}, \cite{Kra}, \cite{Sit1}--\cite{Sit5}.

In this section we show how transmutation method can be used to deduce the solution of the
Cauchy problem for the general Euler--Poisson--Darboux equation
\begin{equation} \label{GEQ142} (B_k)_tu=(\Delta _{\gamma })_x u, \quad u=u(x,t;k),\qquad x\in\mathbb{R}^n_+,\quad t>0,\quad k\in\mathbb{R},
 \end{equation}
\begin{equation} \label{GEQ143} u(x,0;k)=f(x),\quad \quad u_{t} (x,0;k)=0.
\end{equation}

\begin{theorem} Let $f=f(x)$, $x\in\mathbb{R}^n_+$ be twice continuous differentiable function even with respect of each variable. Then for the case $k> n+|\gamma |-1{\kern 1pt} $ the solution of \eqref{GEQ142}--\eqref{GEQ143} is
 \begin{equation}\label{Sol1}
u(x,t;k){=}
\frac{2^{n}\Gamma\left(\frac{k+1}{2}\right)}{\Gamma\left(\frac{k{-}n{-}|\gamma|{+}1}{2}\right)\prod\limits_{i=1}^n{\Gamma\left(\frac{\gamma_i{+}1}{2}\right)}}
\int\limits_{B_{1}^{+} (n)}[\,^{\gamma}T^{ty}f(x)](1{-}|y|^2)^{\frac{k{-}n{-}|\gamma|{-}1}{2}}y^\gamma dy.
\end{equation}
The solution of the problem \eqref{GEQ142}--\eqref{GEQ143} for the $k{=}n{+}|\gamma |{-}1$ is the weighted spherical mean  $M^\gamma_t[f(x)]$.
\end{theorem}

\begin{proof}
Using Lemma 4 we obtain  that the
 weighted spherical mean
 of any twice continuously differentiable function
$f=f(x)$ even with respect to each of the independent variables $x_{1} ,\ldots ,x_{n} $ on $\mathbb{R}_{+}^n$ satisfies the general  Euler--Poisson--Darboux equation
$$
(B_{k} )_{t} M^\gamma_t[f(x)]=(\Delta _{\gamma } )_{x} M^\gamma_t[f(x)]{\kern 1pt} ,\quad k=n+|\gamma |-1{\kern 1pt}
$$
and initial conditions (see \eqref{MeanCond})
$$ M^\gamma_0[f(x)]=f(x),\quad M^\gamma_t[f(x)]\biggr|_{t=0}=0.
$$
It means the the weighted spherical mean $M^\gamma_t[f(x)]$ is the solution of the problem \eqref{GEQ142}--\eqref{GEQ143} for the $k=n+|\gamma |-1$.

In order to obtain the solution of \eqref{GEQ142}--\eqref{GEQ143}
for   $k>n+|\gamma |-1$ we will use the method of descent.
First, we will seek solution of the Cauchy problem \eqref{GEQ142}--\eqref{GEQ143}
for the case $k>n+|\gamma|$.

Let $\gamma'=(\gamma_1,...,\gamma_n,\gamma_{n+1}')$, $\gamma_{n+1}'>0$, $x'=(x_{1} ,...,x_{n+1})\in\mathbb{R}^{n+1}_+$
and
$$
(\Delta_{\gamma'})_{x'}=(B_{\gamma_1})_{x_{1} } +...+(B_{\gamma_n})_{x_{n} } +(B_{\gamma_{n+1}'})_{x_{n+1} }.
$$
 Consider the equation of the type \eqref{GEQ142}
$$
(B_k)_tu=(\Delta_{\gamma'})_{x'}u,\quad u=u(x',t;k),\quad x'\in\mathbb{R}^{n+1}_+,\quad t>0
$$
with the initial conditions
$$
u(x',0;k)=f_{1}(x') ,\quad \quad u_{t} (x',0;k)=0.
$$
When $k=n+|\gamma'|=n+|\gamma|+\gamma_{n+1}'$ the
 weighted spherical mean $M^\gamma_t[f_1(x')]$  is a solution of this Cauchy problem:
$$ u(x',t;k)=
$$
\begin{equation} \label{GEQ144}
=\frac{1}{|S_{1}^{+} (n+1)|_{\gamma'} } \int_{S_{1}^{+} (n+1)} [\,^{\gamma_1}T^{ty_{1} } ...\,\,^{\gamma_n}T^{ty_{n} } \,^{\gamma_{n+1}'}T^{ty_{n+1} } f_1(x)] (y')^{\gamma'} dS_{y'},
 \end{equation}
 $$
  y'=(y_1,...,y_n,y_{n+1}')\in\mathbb{R}^{n+1}_+,
 $$
 $$
 |S_{1}^{+} (n+1)|_{\gamma'}=\frac{\prod\limits_{i=1}^n{\Gamma\left(\frac{\gamma_i{+}1}{2}\right)}\Gamma\left(\frac{\gamma_{n+1}'{+}1}{2}\right)}{2^{n}\Gamma\left(\frac{n{+}1{+}|\gamma|+\gamma_{n+1}'}{2}\right)}=\frac{\prod\limits_{i=1}^n{\Gamma\left(\frac{\gamma_i{+}1}{2}\right)}\Gamma\left(\frac{k-n-|\gamma|{+}1}{2}\right)}{2^{n}\Gamma\left(\frac{k+1}{2}\right)}.
 $$
Let us put $f_{1} (x_{1} ,...,x_{n} ,0)=f(x_{1} ,...,x_{n} ),$ where
 $f$ is the function which appears in  initial conditions \eqref{GEQ143}. In this way the $u$ defined by \eqref{GEQ144} becomes a function only of $x_{1} ,...,x_{n} $ which satisfies equation \eqref{GEQ142} and  initial conditions \eqref{GEQ143}.
We have
$$
u(x,t;k)=\frac{1}{|S_{1}^{+} (n+1)|_{\gamma'} } \int\limits_{S_{1}^{+} (n+1)} [\,^{\gamma}T^{ty}f(x)](y')^{\gamma'} dS_{y'},\qquad\gamma_{n+1}'=k-n-|\gamma|.
$$
Now we rewrite the integral over the part of the sphere $S_{1}^{+} (n+1)$ as an integral over the part of ball $B_{1}^{+} (n){=}\{y{\in}\mathbb{R}^{n}_+{:}\sum\limits_{i=1}^{n} y_i^2\leq 1\}$.
We write the surface integral over multiple integral:
$$
\int\limits_{S_{1}^{+} (n+1)} [\,^{\gamma}T^{ty}f(x)](y')^{\gamma'} dS_{y'}
=\int\limits_{B_{1}^{+} (n)}[\,^{\gamma}T^{ty}f(x)](1-y_1^2-...-y_{n}^2)^{\frac{\gamma_{n+1}'-1}{2}}y^\gamma dy=
$$
$$
=\int\limits_{B_{1}^{+} (n)}[\,^{\gamma}T^{ty}f(x)](1-|y|^2)^{\frac{k-n-|\gamma|-1}{2}}y^\gamma dy,
$$
where $B_{1}^{+} (n)$ is a projection of the  $S_{1}^{+} (n+1)$ on the equatorial plane $x_{n+1}=0$.
We have
\begin{equation}\label{Sol1}
u(x,t;k){=}\frac{2^{n}\Gamma\left(\frac{k+1}{2}\right)}{\prod\limits_{i=1}^n{\Gamma\left(\frac{\gamma_i{+}1}{2}\right)}\Gamma\left(\frac{k{-}n{-}|\gamma|{+}1}{2}\right)}
 \int\limits_{B_{1}^{+} (n)}[\,^{\gamma}T^{ty}f(x)](1{-}|y|^2)^{\frac{k{-}n{-}|\gamma|{-}1}{2}}y^\gamma dy.
\end{equation}

Although \eqref{Sol1} was obtained as the solution of the problem  \eqref{GEQ142}--\eqref{GEQ143} for the case $k{>}n{+}|\gamma |$ the integral on its right-hand side converges and for $k>n+|\gamma |-1$. We can verify by direct substitution
 \eqref{Sol1} in \eqref{GEQ142}--\eqref{GEQ143}  that \eqref{Sol1} satisfies the differential equation \eqref{GEQ142} and the initial conditions \eqref{GEQ143} for all values of $k$ which are greater than $n+|\gamma |-1$.
Let show it. Changing coordinates from $y$ to $y/t$ and using that $(B_{\gamma_i})_{x_i}\,^{\gamma_i}T_{x_i}^{y_i}=(B_{\gamma_i})_{y_i}\,^{\gamma_i}T_{x_i}^{y_i}$ (see \cite{Levitan}) we obtain
$$
I=(\triangle_\gamma)_x\int\limits_{B_{1}^{+} (n)}[\,^{\gamma}T^{ty}f(x)](1-|y|^2)^{\frac{k-n-|\gamma|-1}{2}}y^\gamma dy=
$$
$$
=\sum\limits_{i=1}^n (B_{\gamma_i})_{x_i}\int\limits_{B_{1}^{+} (n)}[\,^{\gamma}T^{ty}f(x)](1-|y|^2)^{\frac{k-n-|\gamma|-1}{2}}y^\gamma dy=
$$
$$
=t^{1-k}\sum\limits_{i=1}^n \int\limits_{B_{t}^{+} (n)}[(B_{\gamma_i})_{x_i}\,^{\gamma}T^{y}f(x)](t^2-|y|^2)^{\frac{k-n-|\gamma|-1}{2}}y^\gamma dy=
$$
\begin{equation}\label{RH}
=t^{1-k}\sum\limits_{i=1}^n \int\limits_{B_{t}^{+} (n)}[(B_{\gamma_i})_{y_i}\,^{\gamma}T^{y}f(x)](t^2-|y|^2)^{\frac{k-n-|\gamma|-1}{2}}y^\gamma dy,
\end{equation}
where $B_{t}^{+} (n){=}\{y{\in}\mathbb{R}^{n}_+{:}\sum\limits_{i=1}^{n} y_i^2\leq t\}.$

For  integrable over the $\overline{\Omega}\,^+$ functions $w,v{\in}C^2_{ev}(\overline{\Omega}\,^+)$ we have the Green formula \eqref{EQ18}.  By applying formula \eqref{EQ18} to the right--hand side of relation \eqref{RH}, we get
$$
I=t^{1-k}\sum\limits_{i=1}^n\int\limits_{S^+_t(n)}\left[\frac{\partial}{\partial y_i}\,^{\gamma}T^{y}f(x)\right](t^2-|y|^2)^{\frac{k-n-|\gamma|-1}{2}}\,\cos(\vec{\nu},\vec{e}_i)\,y^\gamma \,dS,
$$
where $\vec{e}_i$ is the direction of the axis $Oy_i$, $i = 1,...,n$, and thus $\cos(\vec{\nu},\vec{e}_i)=\frac{y_i}{t}$.
Now, by using the fact that the direction of the outward normal to the boundary of a ball with center the origin coincides with the direction of the position vector of the point on the ball, we obtain the relation
$$
I=\frac{1}{t^k}\frac{\partial}{\partial t}t^k\frac{\partial}{\partial t}\int\limits_{B^+_1(n)}\left[\,^{\gamma}T^{ty}f(x)\right](1-|y|^2)^{\frac{k-n-|\gamma|-1}{2}}\,y^\gamma \,dy.
$$
Given that $\frac{1}{t^k}\frac{\partial}{\partial t}t^k\frac{\partial}{\partial t}=(B_k)_t$ and \eqref{RH} we have
$$
(\triangle_\gamma)_x\int\limits_{B_{1}^{+} (n)}[\,^{\gamma}T^{ty}f(x)](1-|y|^2)^{\frac{k-n-|\gamma|-1}{2}}y^\gamma dy=
$$
$$
=(B_k)_t\int\limits_{B^+_1(n)}\left[\,^{\gamma}T^{ty}f(x)\right](1-|y|^2)^{\frac{k-n-|\gamma|-1}{2}}\,y^\gamma \,dy.
$$
It means that $u(x,t;k)$ defined by the formula \eqref{Sol1} indeed satisfies equation \eqref{GEQ142}  for $k{>}n+|\gamma |-1$.
Validity of the first and the second initial conditions follows from the formulas (5.20) and (5.21) from \cite{Levitan} respectively.
\end{proof}

\begin{theorem}
 Let $f=f(x)$, $f\in C^{\left[\frac{n+|\gamma|-k}{2}\right]+2}_{ev}$. Then  the solution of \eqref{GEQ142}--\eqref{GEQ143}
for $k<n+|\gamma |-1$, $k\neq -1,-3,-5,...$
\begin{equation}\label{34}
u(x,t;k)=t^{1-k}\,\left(\frac{\partial}{t \partial t}\right)^m(t^{k+2m-1}u(x,t;k+2m)),
\end{equation}
where $m$ is a minimum integer such that $m\geq \frac{n+|\gamma|-k-1}{2}$ and $u(x,t;k+2m)$ is the solution of the Cauchy problem
\begin{equation}\label{32}
   (B_{k+2m})_t u=(\Delta_\gamma)_x u,
\end{equation}
\begin{equation}\label{33}
u(x,0;k+2m)=\frac{f(x)}{(k+1)(k+3){...}(k+2m-1)},\qquad u_t(x,0;k+2m)=0.
\end{equation}
\end{theorem}
\begin{proof}
 In order to proof that \eqref{34} is the solution of \eqref{GEQ142}--\eqref{GEQ143} when $k{<}n+|\gamma|{-}1$, $k{\neq} {-}1,{-}3,{-}5,{...}$ we will use the recursion formulas \eqref{Rec1} and \eqref{Rec2}. Let choose minimum integer $m$ such that $k+2m\geq n+|\gamma|-1$. Now we can write the solution of the Cauchy problem
$$
   (B_{k+2m})_t u=(\Delta_\gamma)_x u,
$$
$$
u(x,0;k+2m)=g(x),\qquad u_t(x,0;k+2m)=0,\quad g\in C^2_{ev}
$$
by \eqref{Sol1}. We have
$$
u(x,t;k+2m)=
$$
$$
=\frac{2^{n}\Gamma\left(\frac{k+2m+1}{2}\right)} {\prod\limits_{i=1}^n{\Gamma\left(\frac{\gamma_i{+}1}{2}\right)}\Gamma\left(\frac{k+2m-n-|\gamma|{+}1}{2}\right)} \int\limits_{B_{1}^{+} (n)}[\,^{\gamma}T^{ty}g(x)](1-|y|^2)^{\frac{k+2m-n-|\gamma|-1}{2}}y^\gamma dy,
$$
and, using \eqref{Rec1} we obtain
$$
t^{k+2m-1}u(x,t;k+2m)=u(x,t;2-k-2m).
$$
Applying \eqref{Rec2} to the last formula  $m$ times we get
$$
\left(\frac{\partial}{t \partial t}\right)^m(t^{k+2m-1}u(x,t;k+2m)=u(x,t;2-k).
$$
Applying again \eqref{Rec1} we can write
\begin{equation}\label{34}
u(x,t;k)=t^{1-k}\,\left(\frac{\partial}{t \partial t}\right)^m(t^{k+2m-1}u(x,t;k+2m)),
\end{equation}
which gives the solution of the  \eqref{32}. Now we obtain the function $g$ such that the \eqref{33} is true.
From \eqref{34} we have asymptotic relation
$$
u(x,t;k)=(k+1)(k+3)...(k+2m-1)u(x,t;k+2m)+C\,t\,u(x,t;k+2m)+O(t^2),\quad t\rightarrow 0,
$$
where $C$ is a constant. Therefore, if
$$
g(x)=\frac{f(x)}{(k+1)(k+3){...}(k+2m-1)}
$$
then $u(x,t;k)$ defined by \eqref{34} satisfies the initial conditions \eqref{GEQ143}.

Let us recall that for $u(x,t;k+2m)$ to be a solution of \eqref{32}--\eqref{33} it is sufficient that $f\in C^2_{ev}$.
 In order to be able to carry out the construction \eqref{34}, it is sufficient to require that
$f\in C^{\left[\frac{n+|\gamma|-k}{2}\right]+2}_{ev}$.
\end{proof}

\begin{theorem} If $f$ is B--polyharmonic of order $\frac{1-k}{2}$ and even with respect to each variable then one of the solutions of the Cauchy problem \eqref{32}--\eqref{33}  for the  $k{=}{-}1,{-}3,{-}5,{...}$ is given by
 \begin{equation}\label{Ex1}
 u(x,t;k)=f(x),\qquad k=-1,
\end{equation}
\begin{equation}\label{Ex2}
 u(x,t;k)=f(x)+\sum\limits_{h=1}^{-\frac{k+1}{2}}\frac{\Delta^h_\gamma f}{(k+1)...(k+2h-1)}\,\frac{t^{2h}}{2\cdot 4\cdot .... \cdot2h},\qquad k=-3,-5,...
\end{equation}
 \end{theorem}
\begin{proof}
Let us first take $k=-1$ and assume that $\lim\limits_{t\rightarrow 0}\frac{\partial^2u(x,t;-1)}{\partial t^2}$ exists. Let $t\rightarrow 0$ in
$$
(\Delta_\gamma)_x u^{-1}(x,t;k)=\frac{\partial^2 u(x,t;-1)}{\partial t^2}-\frac{1}{t}\frac{\partial u(x,t;-1)}{\partial t},
$$
i.e
$$
(\Delta_\gamma)_x u(x,0;-1){=}\lim\limits_{t\rightarrow 0}\frac{\partial^2 u(x,t;-1)}{\partial t^2}{-}\lim\limits_{t\rightarrow 0}\frac{1}{t}\frac{\partial u(x,t;-1)}{\partial t}{=}0.
$$
We find that  $(\Delta_\gamma)_x u(x,0;-1)=0$ which shows that $f$ must be B-harmonic. So the function $f$ satisfies \eqref{32}--\eqref{33}  for the  $k{=}{-}1$.

When $k=-3$ we have
$$
\lim\limits_{t\rightarrow 0}\frac{\partial^2 u(x,t;-3)}{\partial t^2}=\lim\limits_{t\rightarrow 0}\frac{1}{t}\frac{\partial u(x,t;-3)}{\partial t}.
$$
From the general form of the  Euler--Poisson--Darboux equation for $k{=}{-}3$ we obtain
$$
\lim\limits_{t\rightarrow 0}(\Delta_\gamma)_x u(x,t;-3){=}\lim\limits_{t\rightarrow 0}\frac{\partial^2 u(x,t;-3)}{\partial t^2}-3\lim\limits_{t\rightarrow 0}\frac{1}{t}\frac{\partial u(x,t;-3)}{\partial t}{=}-2\lim\limits_{t\rightarrow 0}\frac{1}{t}\frac{\partial u(x,t;-3)}{\partial t}.$$
It is follows from \eqref{Rec2} that
$$\frac{1}{t}\frac{\partial u(x,t;-3)}{\partial t}=u(x,t;-1)
$$
hence
 \begin{equation}\label{35}
   \lim\limits_{t\rightarrow 0}(\Delta_\gamma)_x u(x,t;-3)=-2u(x,0;-1).
 \end{equation}
If the limit $\lim\limits_{t\rightarrow 0}\frac{\partial^4 u(x,t;-3)}{\partial t^4}$ exists  and all odd derivatives of $u(x,t;-3)$ tend to zero when $t\rightarrow0$, then $\lim\limits_{t\rightarrow 0}\frac{\partial^2 u(x,t;-1)}{\partial t^2}$ also exists. Therefore, $\lim\limits_{t\rightarrow 0}(\Delta_\gamma)_x u(x,t;-1)=0$  and by \eqref{35} we have $\lim\limits_{t\rightarrow 0}(\Delta_\gamma)_x^2 u(x,t;-3)=0$. This remark can be easily generalized to include all the exceptional values.
So, in this case a solution of Cauchy problem for the general form of the Euler--Poisson--Darboux equation for the case  $k{=}{-}3,{-}5,{...}$ is given by the formula
 $$
 u(x,t;k)=f(x)+\sum\limits_{h=1}^{-\frac{k+1}{2}}\frac{\Delta^h_\gamma f}{(k+1)...(k+2h-1)}\,\frac{t^{2h}}{2\cdot 4\cdot .... \cdot 2h},\qquad k=-3,-5,...
 $$
 and as we proved earlier $ u(x,t;-1)=f(x)$.
\end{proof}

\section{Solution of the  singular Cauchy problem using the Hankel transform}

In this section we are looking for the solution $u\in S_{ev}'(\mathbb{R}^{n}_{+})\times C^2(0,\infty)$\footnote{ Notation $u\in S_{ev}'(\mathbb{R}^{n}_{+})\times C^2(0,\infty)$ means that $u(x,t;k)$ belongs to $S_{ev}'(\mathbb{R}^{n}_{+})$ by variable $x$ and belongs to $C^2(0,\infty)$ by variable $t$.}
 of the
\begin{equation} \label{GEQ1421} (B_k)_tu{=}(\Delta _{\gamma })_x u, \quad u=u(x,t;k),\quad x\in\mathbb{R}^n_+,\quad t>0,
 \end{equation}
\begin{equation} \label{GEQ1431} u(x,0;k)=f(x),\quad \quad u_{t} (x,0;k)=0.
\end{equation}
when $f(x)\in S_{ev}'(\mathbb{R}^{n}_{+})$, $k\in\mathbb{R}\setminus\{-1,-3,-5,...\}$.

\begin{theorem} The solution $u\in S_{ev}'(\mathbb{R}^{n}_{+})\times C^2(0,\infty)$ of the  \eqref{GEQ1421}--\eqref{GEQ1431} when $k\neq -1,-3,-5,...$ is defined by the formula
\begin{equation}\label{Sol0}
u(x,t;k)=\frac{ 2^n t^{1-k}\Gamma\left(\frac{k+1}{2}\right)}{\Gamma\left(\frac{k-n-|\gamma|+1}{2}\right)\prod\limits^n_{i=1}\Gamma\left(\frac{\gamma_i+1}{2}\right)}((t^2-|x|^2)_{+,\gamma}^{\frac{k-n-|\gamma|-1}{2}}*f(x))_\gamma.
\end{equation}
The solution  \eqref{Sol0} is unique for $k\geq0$ and not unique for $k<0$.
\end{theorem}
\begin{proof}
Applying multidimensional Hankel transform to \eqref{GEQ1421} with respect to the  variables $x_1,...,x_n$ only and using \eqref{SaprBessel} we obtain
\begin{equation}\label{22}
\left(|\xi|^2+\frac{\partial^2}{\partial t^2}+\frac{k}{t}\frac{\partial}{\partial t}\right)\widehat{u}(\xi,t)=0,
\end{equation}
\begin{equation}\label{23}
\lim\limits_{t\rightarrow 0}\widehat{u}(\xi,t;k)=\widehat{f}(\xi),\qquad \lim\limits_{t\rightarrow 0}\frac{\partial \widehat{u}(\xi,t;k)}{\partial t}=0,
\end{equation}
where $\xi{=}(\xi_l,\xi_2,{...},\xi_n){\in}\mathbb{R}^n_+$ corresponds to $x{=}(x_1,{...},x_n){\in}\mathbb{R}^n_+$, $|\xi|^2{=}\xi^2_1{+}\xi_2^2{+}{...}{+}\xi_n^2$,
$$
\widehat{u}(\xi,t;k)=(F_B)_x[u(x,t;k)](\xi)=\int\limits_{\mathbb{R}^n_+}u(x,t;k)\,\mathbf{j}_\gamma(x;\xi)x^\gamma dx
$$
and $\widehat{f}(\xi)=F_B[f](\xi)$.

In \cite{Bresters1} the solution  $\widehat{G}^k(\xi,t)$ of the Cauchy problem
$$
\left(|\xi|^2+\frac{\partial^2}{\partial t^2}+\frac{k}{t}\frac{\partial}{\partial t}\right)\widehat{G}^k(\xi,t)=0,
$$
$$
\lim\limits_{t\rightarrow 0}\widehat{G}^k(\xi,t)=1,\qquad \lim\limits_{t\rightarrow 0}\frac{\partial \widehat{G}^k(\xi,t)}{\partial t}=0
$$
was obtained and it has the form
\begin{equation}\label{221}
    \widehat{G}^k(\xi,t)=j_{\frac{k-1}{2}}(|\xi| t),
\end{equation}
for $ k\geq 0,$
\begin{equation}\label{222}
\widehat{G}^k(\xi,t)=j_{\frac{k-1}{2}}(|\xi| t)+A t^{\frac{1-k}{2}}J_{\frac{1-k}{2}}(|\xi| t),
\end{equation}
for $ k<0,\qquad k\neq -1,-3,-5,...$,
\begin{equation}\label{223}
        \widehat{G}^k(\xi,t)=B t^{\frac{1-k}{2}}J_{\frac{1-k}{2}}(|\xi| t)-\frac{\pi 2^{\frac{k-1}{2}}}{\Gamma\left(\frac{1-k}{2}\right)}\, (|\xi| t)^{\frac{1-k}{2}} Y_{\frac{1-k}{2}}(|\xi| t),
\end{equation}
for $k=-1,-3,-5,...$.

In \eqref{221}--\eqref{223} $A$ and $B$ are arbitrary complex numbers and $Y_\nu(z)$ is a Bessel functions
of the  the second kind. The solutions \eqref{222}, \eqref{223} depend on the constants $A$ and $B$  since they are is not unique (see \cite{Bresters1}).
When $\widehat{G}^k(\xi,t)$ is found, the solution of \eqref{22}--\eqref{23} is
$$
\widehat{u}(\xi,t;k)=\widehat{G}^k(\xi,t)\cdot \widehat{f}(\xi)
$$
and the solution of \eqref{GEQ1421}--\eqref{GEQ1431} is then given by
$$
u(x,t;k)=((F^{-1}_B)_\xi[\widehat{G}^k(\xi,t)]*f(x))_\gamma=({G}^k(x,t)*f(x))_\gamma.
$$

We are looking for the solution when $k\neq -1,-3,-5,...$ and when $A=0$. The  obtained solution will be unique for $k\geq 0$ and will be one of the possible solutions for $k<0$, $k\neq -1,-3,-5,...$. So we are interested in case when $\widehat{G}^k(\xi,t)=j_{\frac{k-1}{2}}(|\xi| t)$
Using \eqref{FBOF} we can find $(F^{-1}_B)_\xi[j_{\frac{k-1}{2}}(|\xi| t)](x)$:
$$
{G}^k(x,t)=(F^{-1}_B)_\xi[j_{\frac{k-1}{2}}(|\xi|\cdot t)](x)=\frac{ 2^n t^{1-k}\Gamma\left(\frac{k+1}{2}\right)}{\Gamma\left(\frac{k-n-|\gamma|+1}{2}\right)\prod\limits^n_{i=1}\Gamma\left(\frac{\gamma_i+1}{2}\right)}(t^2-|x|^2)_{+,\gamma}^{\frac{k-n-|\gamma|-1}{2}}.
$$
Then the solution of the \eqref{GEQ1421}--\eqref{GEQ1431} has the form
\begin{equation}\label{Sol}
u(x,t;k)=\frac{ 2^n t^{1-k}\Gamma\left(\frac{k+1}{2}\right)}{\Gamma\left(\frac{k-n-|\gamma|+1}{2}\right)\prod\limits^n_{i=1}\Gamma\left(\frac{\gamma_i+1}{2}\right)}\left((t^2-|x|^2)_{+,\gamma}^{\frac{k-n-|\gamma|-1}{2}}*f(x)\right)_\gamma,
\end{equation}
$$
k\neq -1,-3,-5,...
$$

Since $(t^2-|x|^2)_{+,\gamma}^{\lambda}$ has its support in the interior of the part of the sphere $S_1^+(n)$ when $x_1\geq 0,...,x_n\geq0$, we may conclude that the convolution exists for arbitrary $\varphi(x)\in S'_+$. In \cite{Fox} it was shown that the solution of the singular Cauchy problem \eqref{GEQ1421}--\eqref{GEQ1431} is unique when $k$ is  nonnegative and not unique when  $k$ is negative.
\end{proof}

\begin{corollary}
For $k>n+|\gamma|-1$ when $f\in C^2_{ev}$  the solution of \eqref{GEQ1421}--\eqref{GEQ1431}
exists in the classical sense and given by
\begin{equation}\label{SolCon}
u(x,t;k){=}\frac{ 2^n \Gamma\left(\frac{k+1}{2}\right)}{\Gamma\left(\frac{k-n-|\gamma|+1}{2}\right)\prod\limits^n_{i=1}\Gamma\left(\frac{\gamma_i+1}{2}\right)}
\int\limits_{S^+_1(n)}(1-|y|^2)^{\frac{k-n-|\gamma|-1}{2}}\,^\gamma T^{t y}f(x)y^\gamma dy,
\end{equation}
Which coincides with formula \eqref{Sol1}.
\end{corollary}
\begin{proof} In the case when $k>n+|\gamma|-1$ and $f(x)$ is continuous and even with respect to all variables the integral in \eqref{Sol} exists in the classical sense. So, taking in \eqref{Sol}  usual function $(t^2-|x|^2)^{\lambda}$ instead of the weighted generalized function  $(t^2-|x|^2)_{+,\gamma}^{\lambda}$,  passing to the integral over  the part of the sphere $S_t^+{=}\{x{\in}\mathbb{R}^n_+{:}|x|{<}t\}$ and changing the variables by formula $x=ty$ we obtain \eqref{SolCon}.
\end{proof}

\section{Case when the $x$ is one-dimensional}

In this section we concentrate on the case when $x$ is one--dimensional. Then problems and constructed above solutions are simplified. For these problems we consider below  some illustrative examples with explicit solution representations and some visual graphs using the Wolfram Alpha.

In this case we have the Cauchy problem
\begin{equation}\label{OneD1}
\frac{\partial^2 u}{\partial x^2}+\frac{\gamma}{x}\frac{\partial u}{\partial x}=\frac{\partial^2 u}{\partial t^2}+\frac{k}{t}\frac{\partial u}{\partial t},
\end{equation}
\begin{equation}\label{OneD2}
    u(x,0;k)=f(x),\qquad \frac{\partial u(x,t;k)}{\partial t}\biggr|_{t=0}=0,\qquad f(x)\in C_{ev}^2(\overline{\mathbb{R}}\,^{1}_{+}).
\end{equation}
When $k>\gamma>0$ the solution of \eqref{OneD1}--\eqref{OneD2} is given by the formula (see \eqref{Sol1})
\begin{equation}\label{SolEx}
u(x,t;k){=}\frac{ 2\Gamma\left(\frac{k+1}{2}\right)}{\Gamma\left(\frac{k-\gamma}{2}\right)\Gamma\left(\frac{\gamma+1}{2}\right)}
\int\limits_0^1(1-y^2)^{\frac{k-\gamma-2}{2}}\,^\gamma T^{ty}f(x)y^\gamma
dy.
\end{equation}
When $k<\gamma$ the solution of \eqref{OneD1}--\eqref{OneD2} is found by the formulas \eqref{34}, \eqref{Ex1} or \eqref{Ex2}.
\newpage
\textbf{Example 1.}
We are looking for the solution of
$$
\frac{\partial^2 u}{\partial x^2}+\frac{\gamma}{x}\frac{\partial u}{\partial x}=\frac{\partial^2 u}{\partial t^2}+\frac{k}{t}\frac{\partial u}{\partial t},\qquad k>\gamma>0,
$$
$$
    u(x,0;k)={j}_{\frac{\gamma-1}{2}}(x),\qquad \frac{\partial u(x,t;k)}{\partial t}\biggr|_{t=0}=0,\qquad f(x)\in C_{ev}^2(\overline{\mathbb{R}}\,^{1}_{+}).
$$

By \eqref{SolEx} we obtain
$$
u(x,t;k){=}\frac{ 2\Gamma\left(\frac{k+1}{2}\right)}{\Gamma\left(\frac{k-\gamma}{2}\right)\Gamma\left(\frac{\gamma+1}{2}\right)}
\int\limits_0^1(1-y^2)^{\frac{k-\gamma-2}{2}}\,^\gamma T^{ty}{j}_{\frac{\gamma-1}{2}}(x)y^\gamma
dy.
$$
The next formula is valid
$$
T^{ty}{j}_{\frac{\gamma-1}{2}}(x)={j}_{\frac{\gamma-1}{2}}(x){j}_{\frac{\gamma-1}{2}}(ty)
$$
and so
$$
u(x,t;k){=}{j}_{\frac{\gamma-1}{2}}(x)\,t^{\frac{1-\gamma}{2}}\,\frac{ 2^{\frac{\gamma+1}{2}}\Gamma\left(\frac{k+1}{2}\right)}{\Gamma\left(\frac{k-\gamma}{2}\right)}\int\limits_0^1(1-y^2)^{\frac{k-\gamma-2}{2}}
J_{\frac{\gamma-1}{2}}(ty)\,y^{\frac{\gamma+1}{2}}
dy.
$$

Using formula 2.12.4.6 from \cite{IR2}
\begin{equation}\label{Prud}
\int\limits_0^a x^{\nu+1}(a^2-x^2)^{\beta-1}J_\nu(cx)dx=\frac{2^{\beta-1}a^{\beta+\nu}}{c^{\beta}}
\Gamma(\beta)J_{\beta+\nu}(ac),
\end{equation}
$$
a>0,\qquad {\rm Re}\,\beta>0,\qquad {\rm Re}\,\nu>-1
$$
we obtain
\begin{equation}\label{Resh}
u(x,t;k){=}{j}_{\frac{\gamma-1}{2}}(x){j}_{\frac{k-1}{2}}(t).
\end{equation}

\begin{figure}[h!]
\center{\includegraphics[width=0.8\linewidth]{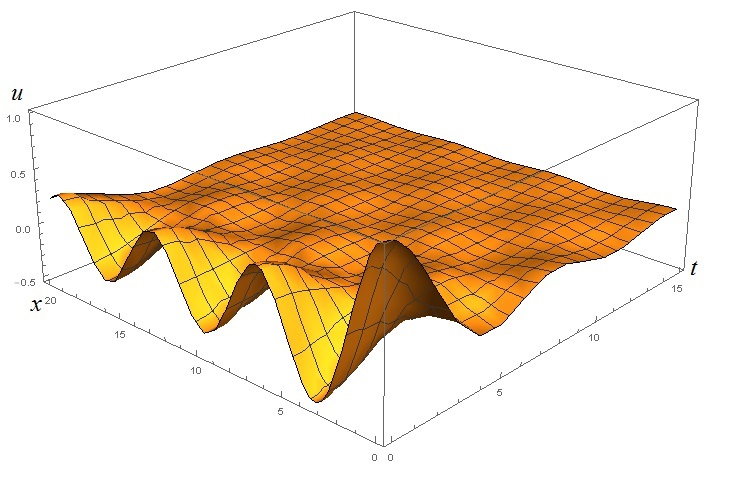}}
\caption{$u(x,t;k){=}{j}_{-\frac{1}{6}}(x){j}_{\frac{3}{4}}(t).$}\label{Ris1}
\end{figure}

\begin{figure}[h!]
\center{\includegraphics[width=0.8\linewidth]{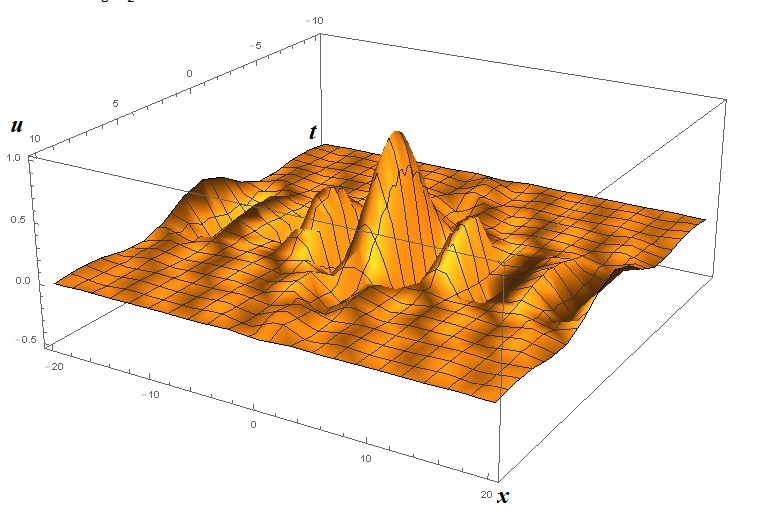}}
\caption{$u(x,t;k){=}{j}_{-\frac{1}{6}}(x){j}_{\frac{3}{4}}(t).$}\label{Ris2}
\end{figure}

The plot of \eqref{Resh} when $k=\frac{5}{2}$ and $\gamma=\frac{2}{3}$ is presented on Figure 1 obtained through the Wolfram$\mid$Alpha.

We can continue the solution to negative values of $x$ and $t$ as an even function. The plot of such continuation is presented on Figure 2.

\newpage

\textbf{Remark.} If we denote
\begin{equation}\label{OOOS}
    \,^{k,\gamma}T^t_xf(x){=}C(\gamma,k)\int\limits_0^1(1-y^2)^{\frac{k-\gamma-2}{2}}T^{ty}f(x)y^\gamma
dy,
\end{equation}
$$
 C(\gamma,k){=}\frac{ 2\Gamma\left(\frac{k+1}{2}\right)}{\Gamma\left(\frac{k-\gamma}{2}\right)\Gamma\left(\frac{\gamma+1}{2}\right)}.
$$
We can consider the operator \eqref{OOOS} as a generalized translation operator (see \cite{Lev3}). For this operator the next property
$$
\,^{k,\gamma}T^t_x{j}_{\frac{\gamma-1}{2}}(x)={j}_{\frac{\gamma-1}{2}}(x){j}_{\frac{k-1}{2}}(t)
$$
is valid.

\textbf{Example 2.} The solution of
$$
\frac{\partial^2 u}{\partial x^2}+\frac{\gamma}{x}\frac{\partial u}{\partial x}=\frac{\partial^2 u}{\partial t^2}+\frac{k}{t}\frac{\partial u}{\partial t},\qquad 1-\gamma\leq k<\gamma,\quad k\neq -1,-3,-5,...,\quad\gamma>\frac{1}{2},
$$
$$
    u(x,0;k)={j}_{\frac{\gamma-1}{2}}(x),\qquad \frac{\partial u(x,t;k)}{\partial t}\biggr|_{t=0}=0.
$$
is given by \eqref{OneD1} where $m=1$:
$$
u(x,t;k)=\frac{1}{t^{k}}\,\frac{\partial}{ \partial t}(t^{k+1}u(x,t;k+2)),
$$
and $u(x,t;k+2)$ is the solution of the Cauchy problem
$$
   (B_{k+2})_t u=(\Delta_\gamma)_x u,
$$
$$
u(x,0;k+2)=\frac{{j}_{\frac{\gamma-1}{2}}(x)}{k+1},\qquad u_t(x,0;k+2)=0.
$$
Using previous example we obtain
$$
u(x,t;k+2){=}\frac{1}{k+1}{j}_{\frac{\gamma-1}{2}}(x){j}_{\frac{k+1}{2}}(t).
$$
and
$$
u(x,t;k)=\,_0F_1\left(;\frac{\gamma +1}{2};-\frac{x^2}{4}\right) \, _0F_1\left(;\frac{k+1}{2};-\frac{t^2}{4}\right).
$$

The plot of \eqref{Resh} when $k=\frac{1}{3}$ and $\gamma=\frac{3}{2}$ is presented on Figure 3 obtained through the Wolfram$\mid$Alpha.
\begin{figure}[h!]
\center{\includegraphics[width=1\linewidth]{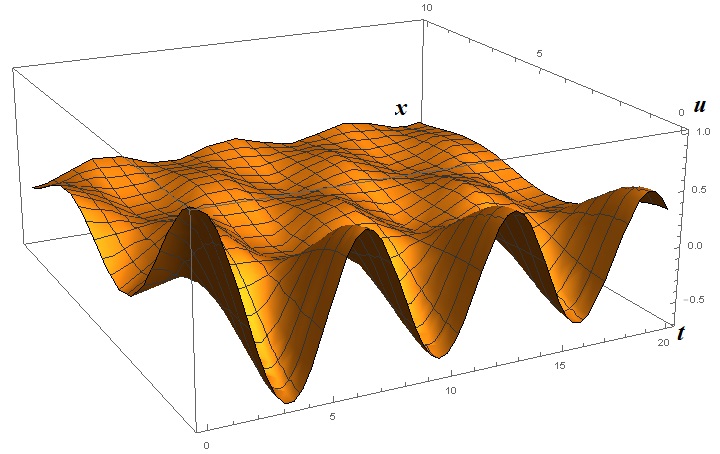}}
\caption{$u(x,t;k){=}\, _0F_1\left(;\frac{2}{3};-\frac{t^2}{4}\right) \, _0F_1\left(;\frac{5}{4};-\frac{x^2}{4}\right).$}\label{Ris3}
\end{figure}

We can continue the solution to negative values of $x$ and $t$ as an even function. The plot of such continuation is presented on Figure 4.
\begin{figure}[h!]
\center{\includegraphics[width=1\linewidth]{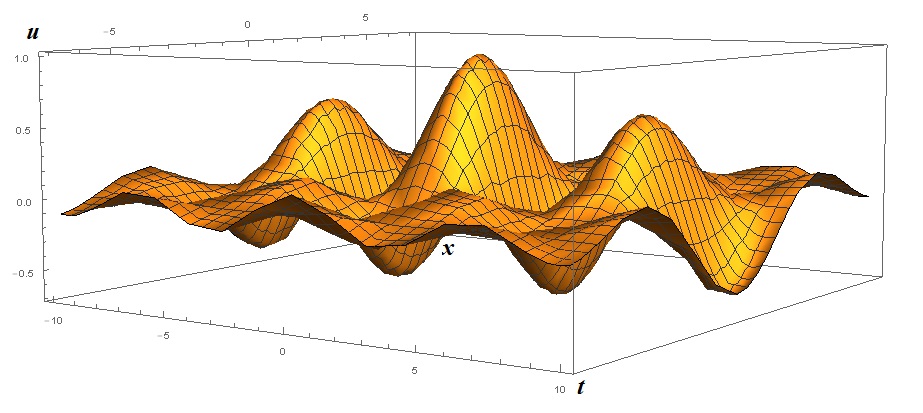}}
\caption{$u(x,t;k){=}\, _0F_1\left(;\frac{2}{3};-\frac{t^2}{4}\right) \, _0F_1\left(;\frac{5}{4};-\frac{x^2}{4}\right).$}\label{Ris4}
\end{figure}

\end{document}